\documentclass[10pt,a4paper]{amsart}
\usepackage[utf8]{inputenc}
\usepackage{amsmath}
\usepackage{amsfonts}
\usepackage{amssymb}
\usepackage{amsthm}
\usepackage{tikz-cd}
\usepackage{hyperref}
\usepackage{cleveref}
\usepackage{url}
\newtheorem{thm}{Theorem}[section]
\newtheorem{prop}[thm]{Proposition}
\newtheorem{lem}[thm]{Lemma}
\newtheorem{defn}[thm]{Definition}
\newtheorem{rem}[thm]{Remark}
\newtheorem{cor}[thm]{Corollary}
\newtheorem{ex}[thm]{Example}
\newcommand{\catname}[1]{\textup{\textbf{#1}}}
\newcommand{\infcatname}[2]{\mathcal{#1}\textup{#2}}
\newcommand{\Set}{\catname{Set}}

\newcommand{\Top}{\catname{Top}}

\newcommand{\biginfcat}{\widehat{\infcatname{C}{at}}_{\infty}}

\newcommand{\Map}{\textup{Map}}
\newcommand{\Fun}{\textup{Fun}}
\newcommand{\Sub}{\textup{Sub}}

\newcommand{\id}{\textup{id}}
\newcommand{\colim}{\textup{colim}\ }
\newcommand{\N}{\textup{N}}
\newcommand{\yon}{\textup{\textbf{y}}}

\newcommand{\Hom}{\textup{Hom}}
\title{Dependent products and 1-inaccessible universes}
\author{Giulio Lo Monaco$^{\dag}$}
\thanks{$^{\dag}$The author was supported by the grant MUNI/A/1186/2018 of Masaryk University, Brno}

\begin{document}

\maketitle

\begin{abstract}
The purpose of this writing is to show that, if we use the definition of elementary $\infty$-topos that has been proposed by Mike Shulman, then the fact that every geometric $\infty$-topos satisfies the required axioms, more specifically the last one of them, is actually something close to a large cardinal assumption. Putting it precisely, we will show that, once a Grothendieck universe has been chosen, the fact that every geometric $\infty$-topos satisfies Shulman's axioms is equivalent to saying that the Grothendieck universe was $1$-inaccessible to start with, a condition which is strictly stronger than just being inaccessible. Moreover, a perfectly analogous result can be shown if instead of geometric $\infty$-toposes our analysis relies on ordinary sheaf toposes. In conclusion, it will be shown that, under stronger assumptions positing the existence of 1-inaccessible cardinals inside the Grothendieck universe, examples of Shulman $\infty$-toposes which are not geometric can be found.
\end{abstract}

\tableofcontents

%\begin{center}
%\begin{tikzcd}
%&& 1 \ar[d]\\
%&& 2 \ar[dr] \ar[dl]\\
%& 3 \ar[dr] \ar[dl] && 4\\
%5 && 6
%\end{tikzcd}
%\end{center}
\section{Preliminaries}
We start by giving a brief account of how geometric $\infty$-toposes are defined and some references for some of the $\infty$-categorical tools that will be used, the theory lying behind them requiring far too much space to be fully presented here. The objects we are going to analyse are often just called $\infty$-toposes, but we prefer to use the adjective “geometric”, firstly in order to remind ourselves that the notion we are handling only constitutes an analog of the $1$-categorical geometric toposes, a.k.a. sheaf toposes, a.k.a. Grothendieck toposes, and secondly to better distinguish them from the other kind of objects that we will be faced with, which throughout this writing will be referred to as Shulman $\infty$-toposes.\\
We assume familiarity with the theory of $\infty$-categories in the language of weak Kan complexes, and we refer the reader to \cite{HTT} for the details thereof. In particular, section 5.5 in \cite{HTT} deals with the theory of presentable $\infty$-categories. We only recall some of the results that we are going to use repeatedly.\\
The statements of \cite{HTT}, Corollary 5.3.4.15 and Remark 5.3.4.16 say:

\begin{prop} \label{retcomp}
For a regular cardinal $\kappa$, the collection of $\kappa$-compact objects in an $\infty$-category is stable under $\kappa$-small colimits and retracts.
\end{prop}

Putting together \cite{HTT}, Propositions 5.5.3.5, 5.5.3.10, 5.5.3.11 and 5.5.6.18, we also obtain the following nice stability properties:

\begin{prop}
Let $\mathcal{C}$ a presentable $\infty$-category. Then:

\begin{itemize}
\item For a simplicial set $S$, the $\infty$-category $\Fun(S,\mathcal{C})$ is presentable.
\item For a functor $p: S \to \mathcal{C}$, the $\infty$-category $\mathcal{C}_{/p}$ is presentable.
\item For a functor as above, the $\infty$-category $\mathcal{C}_{p/}$ is presentable.
\item For an integer $n \geq -2$, the $\infty$-category $\tau_n \mathcal{C}$ of $n$-truncated objects of $\mathcal{C}$ is presentable.
\end{itemize}
\end{prop}

The following two statements, found as \cite{HTT}, Corollary 5.4.1.5 and Proposition 5.3.4.13 respectively, will later on prove of paramount importance in order to give the procedure for our main result a starting kick. We call $\mathcal{S}$ the $\infty$-category of Kan complexes, or $\infty$-groupoids. Since this is equivalent to the $\infty$-category presented by the Quillen model structure on $\Top$, we will interchange freely the words “space” and “Kan complex” for the purposes of this writing. Furthermore, we will use the same notation $\mathcal{P(C)}$ to refer to the category of presheaves of sets whenever $\mathcal{C}$ is an ordinary category, and the $\infty$-category of presheaves of spaces whenever $\mathcal{C}$ is an $\infty$-category. We believe that it will always be unambiguously clear from the context which interpretation should be chosen.

\begin{prop} \label{kancomp}
Let $X$ be a Kan complex, and $\kappa$ an uncountable regular cardinal. Then $X$ is $\kappa$-compact as an object of $\mathcal{S}$ if and only if it is essentially $\kappa$-small, i.e. there is a $\kappa$-small Kan complex $X'$ and a homotopy equivalence $X' \to X$.
\end{prop}

\begin{prop} \label{preshcomp}
Let $\mathcal{C}$ be a presentable $\infty$-category, $S$ a small simplicial set and $f: S \to \mathcal{C}$ a functor. For a regular cardinal $\kappa > |S|$, if for every vertex $s \in S$ the object $f(s)$ is $\kappa$-compact, then $f$ is $\kappa$-compact as an object of $\Fun(S,\mathcal{C})$.
\end{prop}

Finally, \cite{HTT}, Corollary 5.5.2.9 probably deserves to be called one of the most important results in the theory of $\infty$-categories:

\begin{thm}[Adjoint Functor Theorem, presentable version]
Let $f: \mathcal{C} \to \mathcal{D}$ be a functor between presentable $\infty$-categories. Then

\begin{itemize}
\item The functor $f$ has a right adjoint if and only if it preserves small colimits.
\item The functor $f$ has a left adjoint if and only if it is accessible and it preserves small limits.
\end{itemize}
\end{thm}

In conclusion to this section, we mention a couple of useful results about the computation of limits and colimits in $\infty$-categories, found respectively as Proposition 1.2.13.8 and Corollary 5.1.2.3 in \cite{HTT}.

\begin{prop} \label{overcolimit}
Let $\mathcal{C}$ be an $\infty$-category, and $C \in \mathcal{C}$ an object therein. Then the projection $\mathcal{C}_{/C} \to \mathcal{C}$ preserves colimits.
\end{prop}

\begin{prop}
Let $\mathcal{C}$ be an $\infty$-category, $K$ and $S$ simplicial sets. Consider a functor $f: K \to \Fun(S,\mathcal{C})$. Then an extension $f^{\triangleright}: K^{\triangleright} \to \Fun(S,\mathcal{C})$ is a colimit diagram if and only if for every vertex $s \in S$ the induced functor $K^{\triangleright} \to \mathcal{C}$ obtained by evaluating at $s$ is a colimit diagram.
\end{prop}

\section{Geometric $\infty$-toposes}
We know recall the definition of $\infty$-topos. There are quite a few equivalent definitions, which we deem useful to write explicitly.

\begin{thm} \label{deftop}
Given an $\infty$-category $\mathcal{X}$, the following conditions are equivalent:

\begin{enumerate}
\item There is a small $\infty$-category $\mathcal{C}$ such that $\mathcal{X}$ is a left exact accessible localization of $\mathcal{P(C)}$
\item The $\infty$-categorical Giraud's axioms are satisfied:

\begin{itemize}
\item $\mathcal{X}$ is presentable.
\item Colimits in $\mathcal{X}$ are universal.
\item Coproduts in $\mathcal{X}$ are disjoint.
\item Every groupoid object in $\mathcal{X}$ is effective.
\end{itemize}

\item $\mathcal{X}$ is presentable with universal colimits, and the functor

\begin{equation*}
\mathcal{X}^{op} \to \biginfcat
\end{equation*}
taking every object $x \in \mathcal{X}$ to the $\infty$-category $\mathcal{X}_{/x}$ preserves limits.
\item $\mathcal{X}$ is presentable with universal colimits, and for arbitrarily large regular cardinals $\kappa$, the class of relatively $\kappa$-compact morphisms has a classifier.
\item $\mathcal{X}$ is presentable with universal colimits, and for every regular cardinal $\kappa$ such that $\kappa$-compact objects are stable under pullbacks, then the class of relatively $\kappa$-compact morphisms has a classifier.
\end{enumerate}
\end{thm}

The equivalence between (1), (2) and (3) is obtained by combining \cite{HTT}, Theorems 6.1.0.6 and 6.1.3.9, that between these three and (4) is \cite{HTT}, Theorem 6.1.6.8. The discussion immediately preceding this theorem yields that the suitable cardinals mentioned in the statements are those for which the corresponding compact objects are stable under the formation of pullback diagrams. Therefore in order to see the equivalence between (4) and (5) it will suffice to verify that not only do such cardinals exist, but that they can be arbitrarily large. This will be clear with point 2 of \Cref{magic} later on.

\begin{defn}
An $\infty$-category $\mathcal{X}$ is an $\infty$-topos whenever it satisfies the equivalent conditions of \Cref{deftop}.
\end{defn}

Now, for the sake of completeness and motivation, we state once Shulman's proposed definition for an elementary $\infty$-topos, even though thereafter we are going to concentrate on the last axiom alone. Frist, we need a preliminary notion:

\begin{defn}
Given a morphism $f$ in an $\infty$-category (or an ordinary category) $\mathcal{C}$ admitting pullbacks, and denoting with $f^{\ast}$ a pullback functor, we call dependent sum a left adjoint of $f^{\ast}$, and dependent product a right adjoint of $f^{\ast}$, if they exist, and in that case we write

\begin{equation*}
\sum_f \dashv f^{\ast} \dashv \prod_f.
\end{equation*}

If $f: X \to \ast$ is a morphism toward a terminal object, we will also write $\sum_X$ and $\prod_X$ respectively.
\end{defn}

\begin{rem} \label{depexpr}
By the adjoint functor theorem and universality of colimits, every morphism in a geometric $\infty$-topos admits both a dependent sum and a dependent product. Moreover, the dependent sum along $f$ can by expressed as postcomposition with $f$, as a consequence of the pullback property.\\
Dependent products are more complicated, but they have an explicit description as well. In the specific case of a terminal morphism $X \to \ast$ in $\mathcal{S}$ (or in $\Set$), the dependent product of an object $Z$ over $X$ is the space (set) of sections of the structure morphism. This expression can be generalized to some extent in terms of exponentials, assume we already have a way to compute these. With a slight abuse of notation (which ceases to be such in $\mathcal{S}$ and in $\Set$), we denote by $\{ p \} \to X^Z$ a morphism from a terminal object which is adjunct to $p: Z \to X$. Now choose an object $p: Z \to X$ in $\mathcal{C}_{/X}$ and an object $W \in \mathcal{C}$, which pulls back to the object $pr_2: W \times X \to X$ in $\mathcal{C}_{/X}$. Then

\begin{align*}
\Map_X (W \times X,Z) & \simeq \Map(W \times X,Z) \times_{\Map(W \times X,X)} \{ pr_2 \}\\
& \simeq \Map(W, Z^X) \times_{\Map(W,X^X)} \{ \tilde{pr_2} \}\\
& \simeq \Map(W, Z^X \times_{X^X} \{ \id \})
\end{align*}

which means that $\prod_X Z = Z^X \times_{X^X} \{ \id \}$.\\
An analogous expression holds for generic dependent products, where we need to replace exponentials in $\mathcal{C}$ with exponentials as taken in overcategories of the form $\mathcal{C}_{/X}$.
\end{rem}

\begin{defn} \label{elemtopos}
An $\infty$-category $\mathcal{E}$ is called a Shulman $\infty$-topos if it satisfies the following axioms:

\begin{enumerate}
\item[(S1)] $\mathcal{E}$ is finitely complete and cocomplete.
\item[(S2)] $\mathcal{E}$ is locally Cartesian closed, i.e. for every object $X \in \mathcal{E}$ the overcategory $\mathcal{E}_{/X}$ is Cartesian closed.
\item[(S3)] There is a classifier for the class of all monomorphisms, i.e. there is an object $\Omega$ and a monomorphism $\ast \to \Omega$ such that for every object $X \to \mathcal{E}$ the map $\Map(X,\Omega) \to \Sub(X)$ given by pulling back is an equivalence.
\item[(S4)] For every morphism $f \in \mathcal{E}$, there is a class $S$ of morphisms such that $f \in S$, $S$ has a classifier and it is closed under finite limits and colimits as taken in overcategories and under dependent sums and products.
\end{enumerate}
\end{defn}

Every geometric $\infty$-topos satisfies axioms (S1) through (S3), with no particular set-theoretical assumptions (for a proof of this, see for example \cite{mythesis}).\\
We now focus on the axiom (S4), and even more specifically, on a subaxiom thereof, reducing to the minimal statement that will turn out to be equivalent to a large cardinal assumption. Before getting into any set-theoretical definitions, we just say what this subaxiom is.

\begin{defn}
Let $\mathcal{C}$ be an $\infty$-category (or an ordinary category) with pullbacks and admitting dependent sums and products. We say that $\mathcal{C}$ satisfies the axiom (DepProd) if every morphism $f \in \mathcal{C}$ is contained in a class of morphisms $S$ which has a classifier and is closed under dependent products.
\end{defn}

Time has come to present the large cardinals we will need to deal with in the following.

\begin{defn}
Let $\kappa$ be a cardinal. We say that $\kappa$ is 0-inaccessible if it is just inaccessible. Inductively, for an ordinal $\alpha$ we say that $\kappa$ is $\alpha$-inaccessible if it is inaccessible and for every ordinal $\beta < \alpha$ and every cardinal $\lambda < \kappa$ there exists a $\beta$-inaccessible cardinal $\mu$ such that $\lambda \leq \mu < \kappa$.
\end{defn}

\begin{rem}
The above definition simply says that the set of $\beta$-inaccessibles smaller than $\kappa$ is unbounded. Equivalently, since $\kappa$ is regular, it says that the set of $\beta$-inaccessibles smaller than $\kappa$ has cardinality $\kappa$.
\end{rem}

\begin{rem}
In the following, we will mostly be interested in what it means for a Grothendieck universe $\mathcal{U}$ to be 1-inaccessible, i.e. the set of $\mathcal{U}$-small inaccessible cardinals is unbounded. If one wishes to work with the formalism of NBG set theory instead of TG set theory, this can be phrased by saying that inaccessible cardinals form a proper class.
\end{rem}

Now we are ready to state the main result which is the punchline of this writing and will be proven in the next sections.

\begin{thm} \label{main}
Choose a Grothendieck universe $\mathcal{U}$. The following conditions are equivalent:

\begin{enumerate}
\item $\mathcal{U}$ is 1-inaccessible.
\item Every geometric topos satisfies (DepProd).
\item Every geometric $\infty$-topos satisfies (DepProd).
\item Every geometric $\infty$-topos is a Shulman $\infty$-topos.
\end{enumerate}
\end{thm}

We will first prove the implications $(1) \Rightarrow (2)$ and $(1) \Rightarrow (3)$, for which the strategy is exactly the same. Then we will only prove the implication $(3) \Rightarrow (1)$, for which it suffices to just consider the $\infty$-category of Kan complexes, and see that $(2) \Rightarrow (1)$ follows by specializing to discrete Kan complexes, which can be identified with sets. Lastly, the proof $(1) \Rightarrow (4)$ is essentially the same as for $(1) \Rightarrow (3)$, and $(4) \Rightarrow (3)$ is trivial, as it will be reminded at the end of the next section.

\section{1-inaccessibility implies (DepProd)} \label{firsthalf}
One result of paramount importance for all that comes next is the the so-called uniformization theorem. We learned it from \cite{LocPresAccCat} for the case of ordinary categories, and it is readily adapted to the context of $\infty$-categories.

\begin{defn} \label{sharpsmall}
Given two regular cardinals $\lambda < \mu$ we say that $\lambda$ is sharply smaller than $\mu$, and we write $\lambda \triangleleft \mu$, if the two following equivalent conditions are satisfied:

\begin{enumerate}
\item Every $\lambda$-accessible category $\mathcal{C}$ is also $\mu$-accessible.
\item Every $\lambda$-accessible $\infty$-category $\mathcal{D}$ is also $\mu$-accessible.
\item In each $\lambda$-filtered poset every $\mu$-small poset is contained in a $\mu$-small $\lambda$-filtered subposet.
\end{enumerate}
\end{defn}

The equivalence between (1) and (3) is given in \cite{LocPresAccCat}, Theorem 2.11. That (2) implies (1) can be seen just by taking nerves and observing that accessibility of a category is precisely detected on its nerve. The opposite direction goes by steps: first, Proposition 5.3.1.16 in \cite{HTT} allows us to just consider colimit diagrams indexed by nerves of suitable posets in the definition of accessibility of an $\infty$-category. Then we may follow the proof of the implication $(4) \Rightarrow (1)$ of Theorem 2.11 in \cite{LocPresAccCat}, replacing ordinary categories with $\infty$-categories. Finally, we come to the same result by using Corollary 4.2.3.10 in \cite{HTT} and identifying colimits of $\N(I)$ (playing the role of $K$ there) with colimits of $\N(\hat{I})$. This step is rather complicated, but the idea behind it is simple: starting from a diagram $\N(I) \to \mathcal{C}$, we know that 0-cells of $\N(\hat{I})$ are sent to specific colimits in $\mathcal{C}$, then we use their colimit property in order to complete these data with all higher cells, obtaining a diagram that by cofinality has the same colimit as the previous one.\\
We list now a few nice properties of the sharply smaller relation.

\begin{lem} \label{nice}
\begin{enumerate}
\item Given regular cardinals $\lambda < \mu$ such that for all cardinals $\alpha < \lambda$ and $\beta < \mu$ we have $\beta^{\alpha} < \mu$, then $\lambda \triangleleft \mu$.
\item Given arbitrary regular cadinals $\lambda \leq \mu$, then $\lambda \triangleleft (2^{\mu})^+$.
\item Given a set of regular cardinals $(\lambda_i)_{i \in I}$, then there is a regular carinal $\mu$ such that for every $i \in I$ we have $\lambda_i \triangleleft \mu$.
\item Given two regular cardinals $\lambda < \mu$ such that $\mu$ is inaccessible, then $\lambda \triangleleft \mu$.
\end{enumerate}
\end{lem}

\begin{proof}
(1) and (2) are respectively Example 2.13(4) and 2.13(3) in \cite{LocPresAccCat}. Taking $\mu' = \sup_{i \in I} \lambda_i$, we can apply (2) and thus obtain (3) by setting $\mu = (2^{\mu'})^+$. Finally, to prove (4), observe that the definition of inaccessible readily implies (1), hence the result.
\end{proof}

\begin{lem} \label{key}
Let $\lambda \triangleleft \kappa$ be regular cardinals, and let $\mathcal{C}$ be a $\lambda$-accessible $\infty$-category. Then an object $C \in \mathcal{C}$ is $\kappa$-compact if and only if it is a retract of a $\kappa$-small $\lambda$-filtered colimit of $\lambda$-compact objects.
\end{lem}

\begin{proof}
We already know one implication by \Cref{retcomp}, without even using the hypothesis of $\lambda$-filteredness. Conversely, assume that $C$ is $\kappa$-compact. Since $\mathcal{C}$ is $\lambda$-accessible, we can express $C$ as a colimit of $\lambda$-compact objects indexed by the nerve of a $\lambda$-filtered poset $I$. Let $\hat{I}$ be the poset of all $\lambda$-filtered $\kappa$-small subposets of $I$. Given less than $\kappa$ objects in $\hat{I}$, then their union is still $\kappa$-small, therefore condition (3) in \Cref{sharpsmall} yields that it is contained in an object of $\hat{I}$. This means that the poset $\hat{I}$ is $\kappa$-filtered.\\
For each $M \in \hat{I}$, let $B_M$ be a colimit for the diagram indexed by $\N(M)$. Moreover, whenever $M \subseteq M'$ we have canonical maps $B_M \to B_{M'}$ with a contractible choice of higher cells by colimit property. Then all canonical maps $B_M \to C$ (and higher cells given likewise) exhibit $C$ as a colimit for the diagram $\N(\hat{I}) \to \mathcal{C}$ thus defined (that this is true can be checked directly by using the universal property of the indivitual colimits to fill all possible simplices or, alternatively, using Corollary 4.2.3.10 of \cite{HTT} as above). Since $C$ is $\kappa$-compact, $\id_C$ factors as $C \to B_M \to C$ for some $M$, which means that $C$ is a retract of $B_M$ and has therefore the desired description.
\end{proof}

The next theorem is one of the main ingredients to prove the first direction of the implications $(1) \Rightarrow (2)$ and $(1) \Rightarrow (3)$. It is found as Theorem 2.19 in \cite{LocPresAccCat} in the case of ordinary categories. The proof presented here is entirely analogous to the original one.

\begin{thm} [Generalized Uniformization Theorem] \label{gut}
Given a small set of accessible functors $F_i: \mathcal{C}_i \to \mathcal{D}_i$ between presentable $\infty$-categories, there is a regular cardinal $\kappa$ such that each $F_i$ preserves $\kappa$-compact objects. Moreover, this remains true for every other cardinal $\kappa' \triangleright \kappa$.
\end{thm}

\begin{proof}
By \Cref{nice}, there is a regular cardinal $\lambda$ such that all involved $\infty$-categories are $\lambda$-accessible and all functors $F_i$'s are $\lambda$-accessible. Consider all $\lambda$-compact objects of all $\infty$-categories $\mathcal{C}_i$'s. Since they form a small set, there is a regular cardinal $\mu \geq \lambda$ such that all of their images along the respective $F_i$'s are $\mu$-compact. Hence, using again \Cref{nice} we can find a regular cardinal $\kappa \triangleright \lambda$ with the same property. We now show that $\kappa$-compact objects are preserved by each $F_i$.\\
Consider such an object $C \in \mathcal{C}_i$. By \Cref{key} we can write it as a retract of a $\kappa$-small $\lambda$-filtered colimit of $\lambda$-compact objects of $\mathcal{C}_i$. Since $F_i$ is $\lambda$-accessible, such a colimit is preserved, and retractions are obviously preserved as well, thus $F_i(C)$ is a retract of a $\kappa$-small colimit of objects of $\mathcal{D}_i$, all of which are $\mu$-compact by choice of $\mu$. Therefore $F_i(C)$ is $\kappa$-compact.\\
Finally, if we choose a cardinal $\kappa' \triangleright \kappa$, it is also sharply greater than $\lambda$, hence the same proof applies.
\end{proof}

The following statements and proofs apply for the vast majority to both the case of geometric toposes and geometric $\infty$-toposes. Whenever it so happens, we will bracket the symbol $\infty$ to signal that it may be taken into account or ignored at will, yielding analogous results in the two contexts.\\
The following example should shed some light on why uniformization provides such a useful technique when dealing with properties of compact objects.

\begin{ex} \label{magic}
\begin{enumerate}
\item Let $\mathcal{C}$ be an accessible ($\infty$-)category and let $K$ be a small ($\infty$-)category. Consider the set of all projections $\Fun(K,\mathcal{C}) \to \mathcal{C}$ given by evaluating on objects of $K$. Combining \Cref{preshcomp} and \Cref{gut}, we can find a regular cardinal $\kappa$ such that $\kappa$-compact objects in $\Fun(K,\mathcal{C})$ are exactly those functors which take values in $\kappa$-compact objects of $\mathcal{C}$.
\item With the same notation as above, if we instead consider the limit functor $\lim: \Fun(K,\mathcal{C}) \to \mathcal{C}$ we get a cardinal $\kappa$ such that $\kappa$-compact objects are stable under $K$-indexed limits.\\
In particular, if we take $K$ to be the cospan category, we obtain the implication $(5) \Rightarrow (4)$ in \Cref{deftop}.
\item Assume we have a small set of cardinals $(\kappa_i)_{i \in I}$ such that the set of $\kappa_i$-compact objects enjoys the property $P_i$, and also assume that all the properties $P_i$'s are obtained through uniformization. Using \Cref{nice} and \Cref{gut} we find a cardinal $\kappa$ such that $\kappa$-compact object enjoy each of the properties $P_i$'s at the same time.
\item Given any of the preceding cases, suppose we can find a cardinal $\kappa' > \kappa$ such that $\kappa'$ is inaccessible. Then the same result will also hold after replacing $\kappa$ with $\kappa'$, as \Cref{nice} ensures. In particular, if the universe $\mathcal{U}$ is 1-inaccessible to start with, we can always assume that a cardinal found as in the examples above is inaccessible. 
\end{enumerate}
\end{ex}

Our strategy to the proof of $(1) \Rightarrow (2)$ and $(1) \Rightarrow (3)$ will make use of characterization (5) in \Cref{deftop}, which allows us to obtain suitable classes of morphisms that have a classifier and clearly are closed under dependent sums in view of \Cref{depexpr}. The biggest problem will be closure under dependent products. Since they can be obtained in terms of plain exponentials (see \Cref{depexpr} again), we start tackling the issue by just looking at the case of exponentials of objects.

\begin{prop} \label{compset}
In the category $\Set$ of sets and functions, if $X,Y \in \Set$ are $\kappa$-compact for an inaccessible cardinal $\kappa$, then the exponential $X^Y$ is $\kappa$-compact.
\end{prop}

\begin{proof}
Since $\kappa$-compact sets are precisely all those sets whose cardinality is smaller than $\kappa$, the conclusion is immediate in view of the inaccessibility of $\kappa$.
\end{proof}

\begin{prop} \label{compspace}
In the $\infty$-category $\mathcal{S}$ of spaces, if $X,Y \in \mathcal{S}$ are $\kappa$-compact for an inaccessible cardinal $\kappa$, then the exponential $X^Y$ is $\kappa$-compact.
\end{prop}

\begin{proof}
In view of \Cref{kancomp} we know that, for all uncountable cardinals, being $\kappa$-compact is equivalent to being $\kappa$-small in the case of spaces. Therefore it suffices to prove that $X^Y$ has less than $\kappa$ cells, or equivalently less than $\kappa$ $n$-cells in each dimension $n$. The set of $n$-cells is $\Map(Y \times \Delta^n,X)$. Now, since $\kappa$ is inaccessible, for any two cardinals $\mu, \lambda < \kappa$ we have $\mu^{\lambda} < \kappa$, therefore if both spaces are uncountable we obtain set-theoretically the result. If one of them is not, it follows a fortiori.
\end{proof}

For the next proof, we will need a convenient expression for natural transformations between two functors. As one can find in \cite{CWM}, section IX.5 for the case of ordinary categories, or in \cite{GepnerHaugsengNikolaus}, section 5 for the case of $\infty$-categories, the set (space) of natural transformations between two functors $F,G$ starting from a small ($\infty$-)category $\mathcal{C}$ can be expressed as the end

\begin{equation*}
\int_{C \in \mathcal{C}} \Map(F(C),G(C)).
\end{equation*}

Moreover, in both cases the end may be computed as a limit for a diagram indexed by a small ($\infty$-)category whose cardinality is bounded by something which depends on $\mathcal{C}$.

\begin{prop} \label{comppresh}
Assume that the universe $\mathcal{U}$ is 1-inaccessible, and let $\mathcal{C}$ be a $\mathcal{U}$-small ($\infty$-)category. Then there are arbitrarily large inaccessible cardinals $\kappa$ such that $\kappa$-compact objects of $\mathcal{P(C)}$ are stable under exponentiation.
\end{prop}

\begin{proof}
Recalling \ref{preshcomp}, we may choose a cardinal $\mu > |\mathcal{C}|$ and therefore assume that every presheaf taking values in $\mu$-compact sets (spaces) is $\mu$-compact as an object of $\mathcal{P(C)}$. A usage of all four points in \Cref{magic} will provide a cardinal $\kappa$ such that:

\begin{enumerate}
\item All representable presheaves in $\mathcal{P(C)}$ are $\kappa$-compact;
\item $\kappa$-compact objects are stable under binary products;
\item Presheaves on $\mathcal{C}$ are $\kappa$-compact precisely when they take values in $\kappa$-compact sets (spaces);
\item $\kappa$ is inaccessible;
\item Ends of diagrams in $\Set$ (in $\mathcal{S}$) indexed by objects of $\mathcal{C}$ and taking values in $\kappa$-compact objects are themselves $\kappa$-compact.
\end{enumerate}

Now consider two $\kappa$-compact preshesaves $F,G \in \mathcal{P(C)}$, and consider their exponential $F^G$. By (3), it will suffice to prove that for every object $C \in \mathcal{C}$, the set (space) $F^G(C)$ is $\kappa$-compact. In view of the Yoneda lemma and the definition of exponential, we can compute

\begin{align*}
F^G(C) & \simeq \Map(\yon(C), F^G) \\& \simeq \Map(\yon(C) \times G, F) \\& \simeq \int_{D \in \mathcal{C}} \Map(\Map(D,C) \times G(D),F(D)).
\end{align*}

Now, by (1) and (3), $\Map(D,C)$ is $\kappa$-compact for every $D \in \mathcal{C}$, $F(D)$ and $G(D)$ are as well again by (3), therefore $\Map(D,C) \times G(D)$ also is by (2). Using then (4) and \Cref{compset} (\Cref{compspace}), we know that the term inside the end is $\kappa$-compact for every $D \in \mathcal{C}$, which immediately means that the whole end is $\kappa$-compact by applying (5), and the proof is complete. Moreover, this process can be repeated for arbitrarily large inaccessible cardinals, whose existence is granted by the hypothesis that the universe $\mathcal{U}$ is 1-inaccessible.
\end{proof}

\begin{lem} \label{locexp}
Let
\begin{tikzcd}
\mathcal{Y} \ar[r, shift left, "L"] & \mathcal{X} \ar[l, hook, shift left, "i"]
\end{tikzcd}
be a left exact localization between Cartesian closed presentable ($\infty$-)categories. Then the functor $i$ preserves exponentials.
\end{lem}

\begin{proof}
Since the Yoneda embedding of $\mathcal{Y}$ is fully faithful, it suffices to show that there is a natural equivalence between $\Map(A,i(Y^Z))$ and $\Map(A,i(Y)^{i(Z)})$ for every object $A \in \mathcal{Y}$. This follows from the following chain of natural isomorphisms (equivalences)

\begin{alignat*}{3}
& \Map(A,i(Y^Z)) && \simeq \Map(L(A),Y^Z) \quad && \text{by adjunction} \\
& && \simeq \Map(L(A) \times Z,Y) \quad && \text{by exponential property}\\
& && \simeq \Map(L(A \times i(Z)),Y) \quad && \text{by left exactness and fully-faithfulness}\\
& && \simeq \Map(A \times i(Z),i(Y)) \quad && \text{by adjunction}\\
& && \simeq \Map(A,i(Y)^{i(Z)}) \quad && \text{by exponential property}
\end{alignat*}
\end{proof}

\begin{prop} \label{comptop}
Assume that the universe $\mathcal{U}$ is 1-inaccessible, and let $\mathcal{X}$ be a geometric ($\infty$-)topos. Then there are arbitrarily large inaccessible cardinals $\kappa$ such that $\kappa$-compact objects are stable under exponentiation.
\end{prop}

\begin{proof}
Since $\mathcal{X}$ is a geometric ($\infty$-)topos, there is a small ($\infty$-)category $\mathcal{C}$ and a left exact localization of the form

\begin{center}
\begin{tikzcd}
\mathcal{P(C)} \ar[r, shift left, "L"] & \mathcal{X} \ar[l, shift left, hook, "i"].
\end{tikzcd}
\end{center}

The statement is already true in $\mathcal{P(C)}$ since it is exactly \Cref{comppresh}. Enlarging $\kappa$ if necessary by further uniformizing, we may additionally assume that $i$ is $\kappa$-accessible and preserves $\kappa$-compact objects. Now consider two $\kappa$-compact objects $X,Y \in \mathcal{X}$. We wish to show that their exponential $X^Y$ is still $\kappa$-compact. To this end, consider a $\kappa$-filtered colimit of objects $Z_j$'s in $\mathcal{X}$. We have the following chain of isomorphisms (equivalences):

\begin{alignat*}{3}
& \Map(X^Y, \colim Z_j) && \simeq \Map(i(X^Y),i(\colim Z_j)) \quad && \text{by fully-faithfulness}\\
& && \simeq \Map(i(X)^{i(Y)},\colim i(Z_j)) \quad && \text{by \Cref{locexp} and assumption on $i$}\\
& && \simeq \colim \Map(i(X)^{i(Y)},i(Z_j)) \quad && \text{by \Cref{comppresh} and assumption on $i$}\\
& && \simeq \colim \Map(X^Y,Z_j) \quad && \text{by \Cref{locexp} and fully-faithfulness}
\end{alignat*}
\end{proof}

Now that the case of exponentials is settled, we may move forward and start considering more general dependent products. We need a few lemmas first. Given two objects $X \to Y$ and $Z \to Y$ in an overcategory $\mathcal{X}_{/Y}$, we will denote their exponential therein as $(Z^X)_{/Y}$

\begin{lem} \label{pullbackexp}
Let $\mathcal{X}$ be a geometric ($\infty$-)topos and $Y \in \mathcal{X}$ an object, and take two objects $X \to Y, Z \to Y \in \mathcal{X}_{/Y}$. Then for a morphism $f: Y' \to Y$ in $\mathcal{X}$ the pullback of $(Z^X)_{/Y}$ along $f$ is equivalent to $(Z'^{X'})_{/Y'}$, where $X',Z' \in \mathcal{X}_{/Y'}$ are the respective pullbacks of $X$ and $Z$ along $f$.
\end{lem}

\begin{proof}
Since the Yoneda embedding of $\mathcal{X}_{/Y'}$ is fully faithful, it suffices to show that for every object $W$ over $Y'$ the mapping spaces $\Map_{Y'}(W, f^{\ast}(Z^X)_{/Y})$ and $\Map_{Y'}(W, (Z'^{X'})_{/Y'})$ are naturally equivalent. To show this, observe first that calculating binary products in $\mathcal{X}_{/Y}$ is the same as calculating pullbacks over $Y$ in $\mathcal{X}$, and therefore pasting law applied to the double pullback diagram

\begin{center}
\begin{tikzcd}
\bullet \ar[d] \ar[r] & f^{\ast}X \ar[d] \ar[r] & X \ar[d]\\
W \ar[r] & Y' \ar[r, "f"] & Y
\end{tikzcd}
\end{center}

says that the bullet corner is equivalently occupied by $X \times \sum_f W$ of $\sum_f (f^{\ast} X \times W)$ (products computed in the overcategory). Now we have a chain of natural isomorphisms (equivalences)

\begin{alignat*}{3}
& \Map_{Y'} (W,f^{\ast}(Z^X)_{/Y}) && \simeq \Map_Y (\sum_f W, (Z^X)_{/Y}) \quad && \text{by dependent sum property}\\
& && \simeq \Map_Y (X \times \sum_f W, Z) \quad && \text{by exponential property}\\
& && \simeq \Map_Y (\sum_f(f^{\ast} X \times W), Z) \quad && \text{by the observation above}\\
& && \simeq \Map_{Y'} (f^{\ast} X \times W, f^{\ast} Z) \quad && \text{by dependent sum property}\\
& && = \Map_{Y'} (X' \times W, Z') \quad && \text{by definition of $X'$ and $Z'$}\\
& && \simeq \Map_{Y'} (W, (Z'^{X'})_{/Y'}) \quad && \text{by exponential property}
\end{alignat*}
\end{proof}

\begin{lem} \label{overcomp}
Let $\mathcal{C}$ be a presentable ($\infty$-)category with universal colimits, $Y \in \mathcal{C}$ an object and $\lambda$ a regular cardinal. Then an object $p: X \to Y$ in the overcategory $\mathcal{C}_{/Y}$ is $\lambda$-compact if and only if $X$ is $\lambda$-compact in $\mathcal{C}$.
\end{lem}

\begin{proof}
Assume $p: X \to Y$ is $\lambda$-compact in the overcategory, and take $Z$ to be a $\lambda$-filtered colimit of objects $Z_j \in \mathcal{C}$. Then we write

\begin{alignat*}{3}
& \Map(X, \colim Z_j) && \simeq \Map_Y (X, Y \times \colim Z_j) \quad && \text{by definition}\\
& && \simeq \Map_Y (X, \colim Y \times Z_j) \quad && \text{by universality of colimits and \Cref{overcolimit}}\\
& && \simeq \colim \Map_Y (X, Y \times Z_j) \quad && \text{by $\lambda$-compactness}\\
& && \simeq \Map(X,Z_j) \quad && \text{by definition}.
\end{alignat*}

Conversely, assume that $X$ is $\lambda$-compact in $\mathcal{C}$. Consider a $\lambda$-filtered diagram of objects $w_j: W_j \to Y$ in the overcategory. For each one of them, the mapping space $\Map_Y (X,W_j)$ is an equalizer of the diagram

\begin{center}
\begin{tikzcd}
\Map(X,W_j) \ar[r, shift left, "w_j \circ -"] \ar[r, shift right, "const_p", swap] & \Map(X,Y).
\end{tikzcd}
\end{center}

Since in $\Set$ (and in $\mathcal{S}$) $\lambda$-filtered colimits commute with $\lambda$-small limits, this implies the claim.
\end{proof}

\begin{lem} \label{expinover}
Assume that the universe $\mathcal{U}$ is 1-inaccessible, and let $\mathcal{X}$ a geometric ($\infty$-)topos. Then there are arbitrarily large inaccessible cardinals $\kappa$ such that, uniformly for all $\kappa$-compact objects $Y \in \mathcal{X}$, $\kappa$-compact objects in the overcategories $\mathcal{X}_{/Y}$ are stable under exponentiation.
\end{lem}

\begin{proof}
\textit{Step 1}. Suppose $\mathcal{X}$ is $\lambda$-accessible. By \Cref{comptop}, we know that for every $\lambda$-compact object $Y$ there is a cardinal that stabilizes the respective compact objects in $\mathcal{X}_{/Y}$ under exponentiation. Since $\lambda$-compact objects of $\mathcal{X}$ form a small set, by 1-inaccessibility we can take $\kappa$ to be inaccessible and bigger than all these cardinals, thus uniformizing simultaneously in all the overcategories $\mathcal{X}_{/Y}$ with $Y$ being $\lambda$-compact. Therefore we obtain the statement in the special case where $Y$ is $\lambda$-compact. We have to show that it can be extended to all $\kappa$-compact objects.\\
\textit{Step 2}. Uniformizing once more if necessary, we may further assume that $\kappa$-compact objects are stable under pullbacks in $\mathcal{X}$. Now pick a $\kappa$-compact object $Y \in \mathcal{X}$. Since $\mathcal{X}$ is $\lambda$-accessible, by \Cref{key} it can be expressed as a retract of a $\kappa$-small $\lambda$-filtered colimit of $\lambda$-compact objects $Y_j$'s. We denote the retraction at issue with $r: Y' \to Y$. Take two $\kappa$-compact objects $X \to Y, Z \to Y$ of $\mathcal{X}_{/Y}$ and call $X' = r^{\ast} X$ and $Z' = r^{\ast} Z$. By \Cref{overcomp} $X, Z \in \mathcal{X}$ are $\kappa$-compact, and by stability under pullbacks $X'$ and $Z'$ are as well, therefore or every inclusion $a_j: Y_j \to Y'$ in the colimit, $X_j = a_j^{\ast} X'$ and $Z_j = a_j^{\ast} Z'$ are $\kappa$-compact.\\
Now, by \Cref{expinover} we have a double pullback diagram

\begin{center}
\begin{tikzcd}
(Z_j^{X_j})_{/Y_j} \ar[d] \ar[r] & (Z'^{X'})_{/Y'} \ar[d] \ar[r] & (Z^X)_{/Y} \ar[d]\\
Y_j \ar[r, "a_j"] & Y' \ar[r, "r"] & Y
\end{tikzcd}
\end{center}

which, by universality of colimits in $\mathcal{X}$, exhibits $(Z'^{X'})_{/Y'}$ as a colimit of $(Z_j^{X_j})_{/Y_j}$'s, which are all $\kappa$-compact by step 1. Since $(Z^X)_{/Y}$ is now seen to be a retract of $(Z'^{X'})_{/Y'}$ which is a $\kappa$-small colimit of $\kappa$-compact objects, \Cref{key} says that $(Z^X)_{/Y}$ is itself $\kappa$-compact as an object of $\mathcal{X}$, therefore by \Cref{overcomp} also as an object of $\mathcal{X}_{/Y}$, which is exactly what we wanted.
\end{proof}

\begin{thm} \label{almostthere}
Assume that the universe $\mathcal{U}$ is 1-inaccessible. Then in every geomeric ($\infty$-)topos $\mathcal{X}$ there are arbitrarily large inaccessible cardinals $\kappa$ such that the class $S_{\kappa}$ of relatively $\kappa$-compact morphisms is stable under taking dependent products, i.e. the dependent product of a morphism in $S_{\kappa}$ along another morphism in $S_{\kappa}$ is itself in $S_{\kappa}$.
\end{thm}

\begin{proof}
Choose a cardinal $\kappa$ (necessarily inaccessible by the proofs of the preceding propositions) such that $\kappa$-compact objects are stable under pullbacks and the statement of \Cref{expinover} holds. We now show that this is already the cardinal we are looking for.\\
Consider two relatively $\kappa$-compact morphisms $p: Z \to X$ and $f: X \to Y$. We want to show that the dependent product $\prod_f Z \to Y$ of $p$ along $f$ is relatively $\kappa$-compact, therefore consider a $\kappa$-compact object $Y'$ and a morphism $g: Y' \to Y$. We have to prove that $g^{\ast} \prod_f Z$ is $\kappa$-compact.\\
Form a double pullback diagram

\begin{center}
\begin{tikzcd}
Z' \ar[d] \ar[r] & X' \ar[d, "g'"] \ar[r, "f'"] & Y' \ar[d, "g"]\\
Z \ar[r, "p"] & X \ar[r, "f"] & Y.
\end{tikzcd}
\end{center}

We claim that $g^{\ast} \prod_f Z$ is equivalent to $\prod_{f'} Z'$ in the overcategory $\mathcal{X}_{/Y'}$. By Yoneda, it is enough to prove it after composing with $\Map_{Y'}(A,-)$ for any other object $A$ over $Y'$. Observe that, by pasting law, $f^{\ast} \sum_g A$ is equivalent to $\sum_g f'^{\ast} A$. Now we have a chain of isomorphisms (equivalences):

\begin{alignat*}{3}
& \Map_{Y'}(A, g^{\ast} \prod_f Z) && \simeq \Map_Y (\sum_g A, \prod_f Z) \quad && \textup{by dependent sum property}\\
& && \simeq \Map_X (f^{\ast} \sum_g A, Z) \quad && \textup{by dependent product property}\\
& && \simeq \Map_X (\sum_{g'} f'^{\ast} A, Z) \quad && \textup{by the observation above}\\
& && \simeq \Map_{X'} (f'^{\ast} A, g'^{\ast} Z) \quad && \textup{by dependent sum property}\\
& && \simeq \Map_{Y'} (A, \prod_{f'} Z') \quad && \textup{by dependent product property}
\end{alignat*}

which proves the claim true.\\
Since both $f$ and $p$ are relatively $\kappa$-compact and $Y'$ is $\kappa$-compact, then so are $X'$ and $Z'$. We finish the proof recalling that by \Cref{depexpr} $\prod_{f'} Z'$ is computed as $(Z'^{X'})_{/Y'} \times_{(X'^{X'})_{/Y'}} Y'$, which is $\kappa$-compact in view of \Cref{expinover} and stability under pullbacks.
\end{proof}

\begin{cor}[$1) \Rightarrow (2),(3$] \label{univthendep}
Assume that the universe $\mathcal{U}$ is 1-inaccessible. Then for every geometric ($\infty$-)topos $\mathcal{X}$ and every morphism $f \in \mathcal{X}$ there is a class of morphisms $S \ni f$ such that $S$ has a classifier and is closed under dependent products.
\end{cor}

\begin{proof}
Pick a morphism $f: X \to Y$ in $\mathcal{X}$. By universality of colimits, the induced functor $f^{\ast}: \mathcal{X}_{/Y} \to \mathcal{X}_{/X}$ is accessible, therefore we may apply \Cref{gut} to see that it preserves $\kappa$-compact objects for some $\kappa$. Using \Cref{overcomp}, this means precisely that $f$ is relatively $\kappa$-compact. The class $S_{\kappa}$ of relatively $\kappa$-compact morphisms contains $f$ and, by \Cref{magic} and \Cref{deftop}, we may further uniformize and therefore assume that the statements of \Cref{almostthere} hold true and that $S_{\kappa}$ has a classifier, which completes the proof.
\end{proof}

\begin{rem}
By following the steps leading to the proof of \Cref{univthendep}, the rest of axiom (S4) in \Cref{elemtopos} is automatically satisfied. Namely, all classes of morphisms of the form $S_{\kappa}$ are closed under dependent sums and finite limits and colimits in overcategories. Therefore, keeping in mind the brief discussion right after \Cref{elemtopos}, the same proof also shows the implication $(1) \Rightarrow (4)$ in \Cref{main}. Finally, the implication $(4) \Rightarrow (3)$ is trivial, since (DepProd) is a subaxiom of (S4).
\end{rem}

\section{(DepProd) implies 1-inaccessibility}
In this section, we will establish both converse implications of those proven in the previous section. We first focus on the statement $(3) \Rightarrow (1)$ of \Cref{main}, which slightly more complicated than proving $(2) \Rightarrow (1)$, and then observe how our proof of the former already includes in itself a proof of the latter. As the following theorem shows, we may actually start from the apparently weaker assumptions that (DepProd) only hold in the category of sets, or in the $\infty$-category of spaces.

\begin{lem} \label{inacc}
Let $\kappa$ be an uncountable cardinal. Then $\kappa$ is inaccessible if and only if the following condition holds:

\begin{itemize}
\item Given an arbitrary set $I$ such that $|I| < \kappa$ and a family of cardinals $(\alpha_i)_{i \in I}$ such that for all $i \in I, \alpha_i < \kappa$, then we have

\begin{equation*}
\prod_{i \in I}\alpha_i < \kappa.
\end{equation*}
\end{itemize}
\end{lem}

\begin{proof}
Suppose that the condition above is satisfied. To prove that $\kappa$ is regular, take $I$ and $\alpha_i$'s as in the hypothesis, and observe that, if $\alpha_i \geq 2$ for all $i$'s, which of course we can safely assume, then

\begin{equation*}
\sum_{i \in I} \alpha_i \leq \prod_{i \in I} \alpha_i
\end{equation*}

(this is essentially a consequence of Cantor's inequality $2^{\lambda} > \lambda$; for a proof, see \cite{Jech}, formula 5.17, right before Lemma 5.9). Then regularity follows a fortiori by hypothesis.\\
Given a cardinal $\lambda < \kappa$, we may choose a set $I$ of cardinality $\lambda$ and $\alpha_i = 2$. Then the hypothesis says precisely that

\begin{equation*}
2^{\lambda} = \prod_{i \in I} 2 < \kappa
\end{equation*}

so that $\kappa$ is also a strong limit.\\
Conversely, assume that $\kappa$ is inaccessible. Given $I$ and $\alpha_i$'s as in the hypothesis, choose sets $X_i$'s of cardinalities $\alpha_i$ respectively. We need to show that $\prod_{i \in I} X_i$ has cardinality strictly smaller than $\kappa$. Now this can be described as the set of sections of the canonical map $\coprod_{i \in I} X_i \to I$, which is a subset of the exponential $I^{\coprod_{i \in I} X_i}$. In turn, identifying each funcion with its graph, we see that this is a subset of the power set $\mathcal{P}(I \times \coprod_{i \in I} X_i)$. By the assumption that $\kappa$ is a strong limit, it suffices to show that $I \times \coprod_{i \in I} X_i$ is strictly smaller than $\kappa$. Now $\coprod_{i \in I} X_i$ is strictly smaller than $\kappa$ by regularity, therefore the result follows by applying regularity once more, and observing that $I \times Y = \coprod_{i \in I} \{ i \} \times Y$ for an arbitrary set $Y$.
\end{proof}

\begin{thm}[$3) \Rightarrow (1$] \label{deptheninacc}
Fix a universe $\mathcal{U}$. Suppose that for every morphism $f$ in the $\infty$-category $\mathcal{S}$ of $\mathcal{U}$-small spaces there is a class $S$ of morphisms containing $f$ such that $S$ has a classifier and it is closed under dependent products. Then the universe $\mathcal{U}$ is 1-inaccessible.
\end{thm}

\begin{proof}
We must show that, given an arbitrary cardinal $\mu$, there exists an inaccessible cardinal $\kappa$ such that $\kappa > \mu$. Pick a cardinal $\mu$ and consider a discrete space $X$ of cardinality $\mu$. Then choose a suitable class $S$ as in the hypothesis containing the terminal morphism $X \to \Delta^0$. It will have a classifier $t: \bar{U} \to U$. In particular, every fiber of $t$ is an element of $S$, and $X$ is such a fiber. Now consider the full subspace of $U_0 \subseteq U$ spanned by all vertices whose respective fiber is discrete, therefore equivalent to the nerve of a set, and take a pullback diagram

\begin{center}
\begin{tikzcd}
\bar{U}_0 \ar[d, "t_0"] \ar[r] & \bar{U} \ar[d, "t"]\\
U_0 \ar[r] & U.
\end{tikzcd}
\end{center}

For every map of the form $Y \to \Delta^0$ which is contained in $S$ and such that $Y$ is a set, the classifying map factors through $U_0$ by construction, therefore $t_0$ classifies in particular all terminal morphisms in $S$ whose domain is a set.\\
Before pointing to the cardinal we need, let us observe a nice property of $t_0$. Pick a fiber $Z$ of $t_0$ and, for each $z \in Z$, choose a point of $U_0$. This will define a map $Z \to U_0$ which, in turn, yields a map $Y \to Z$. Now, calling $Y_z$ the fiber of $z$, we obtain double pullback diagrams of the form

\begin{center}
\begin{tikzcd}
Y_z \ar[d] \ar[r] & Y \ar[d] \ar[r] & \bar{U}_0 \ar[d]\\
\Delta_0 \ar[r, "z"] & Z \ar[r] & U_0
\end{tikzcd}
\end{center}

which, by universality of colimits in $\mathcal{S}$, means that the map $Y \to Z$ is equivalent to the projection $p: \coprod_{z \in Z} Y_z \to Z$. Now this projection belongs to $S$ by construction, and the terminal map $Z \to \Delta_0$ also belongs to $S$ by assumption, therefore the terminal map from the dependent product $\prod_Z p$ belongs to $S$. But by \Cref{depexpr} this can be expressed as the space of sections of $p$, which in turn may be identified with $\prod_{z \in Z} Y_z$. Since products of discrete spaces are themselves discrete, we obtain the following property:

\begin{itemize}
\item[$\bullet$] Given a set $Z$ which is a fiber of $t_0$ and a set $(Y_z)_{z \in Z}$ of fibers of $t_0$ indexed by $Z$, then the product $\prod_{z \in Z} Y_z$ is also a fiber of $t_0$.
\end{itemize}

Now consider the set $E$ of fibers of $t_0$, and denote the fiber of a vertex $x$ as $F_x$. Since $E$ is a small set, we can define a cardinal $\kappa = \sup_{x \in E} |F_x|$. Now we have that $\mu < \kappa$, because $|X| < |X|^{|X|} = |X^X|$, and $X^X$ is certainly contained in $E$ by the property $\bullet$, since it can be expressed as the product $\prod_{x \in X} X$. Similarly, for any set $F_x \in E$ it is true that $|F_x| < \kappa$. The proof will be complete by showing that $\kappa$ is inaccessible.\\
We may safely assume uncountability just by picking $\mu$ infinite in the first place. In order to complete the proof, we will use the characterization given in \Cref{inacc}. To this end, consider a family $(\alpha_i)_{i \in I}$ such that $|I| < \kappa$ and $\alpha_i < \kappa$ for each $i \in I$. By definition of $\kappa$, there are sets $F_I, F_{\alpha_i} \in E$ such that $|F_I| \geq |I|, |F_{\alpha_i}| \geq \alpha_i$. So there is an injective function $I \to F_I$ that, by adding copies of $X$ if necessary, allows us to index the sets $F_{\alpha_i}$'s over $F_I$. Hence we obtain

\begin{equation*}
\prod_{i \in I} \alpha_i \leq \prod_{i \in F_I} |F_{\alpha_i}| = |\prod_{i \in F_I} F_{\alpha_i}| < \kappa
\end{equation*}

since, using the property $\bullet$ again, we know that the set $\prod_{i \in F_I} F_{\alpha_i}$ belongs to $E$, which is what we wanted.
\end{proof}

\begin{rem}[$2) \Rightarrow (1$] The proof of \Cref{deptheninacc} proceeds entirely on plain sets, the only exception being the first part when we reduce the classifier $t$ to a more restricted classifier $t_0$ only working for sets. Leaving out this first step, the proof also applies in $\Set$, hence becoming precisely the statement $(2) \Rightarrow (1)$ in \Cref{main}.
\end{rem}

\section{Products are stronger than sums}
The combination of both direction of $(1) \Leftrightarrow (2)$ as well as $(1) \Leftrightarrow (3)$ in \Cref{main} says that dependent products are a somewhat stronger notion than dependent sums. This does not mean that closure of a class of morphisms under the former already implies closure under the latter, but that the existence of nice classes of morphisms which are closed under dependent products implies that they can always be extended to classes which are closed under both dependent products and sums.

\begin{defn}
Let $\mathcal{C}$ be a ($\infty$-)category with pullbacks and admitting dependent sums and products. We say that $\mathcal{C}$ satisfies the axiom $(Dep)$ if every morphism $f \in \mathcal{C}$ is contained in a class of morphisms $S$ which has a classifier and is closed under dependent products and sums.
\end{defn}

\begin{cor} \label{notstronger}
Assume that $(DepProd)$ holds in the $\infty$-category of spaces (or in the category of sets). Then the apparently stronger axiom $(Dep)$ already holds in every geometric ($\infty$-)topos. Moreover, an arbitrary class of morphisms $S$ in a geometric ($\infty$-)topos, which has a classifier and is closed under dependent products can be extended to a class of morphisms $S'$ which has a classifier and is closed under dependent products and sums.
\end{cor}

\begin{proof}
Let us choose a geometric ($\infty$-)topos $\mathcal{X}$. Since $(DepProd)$ holds, an arbitrary morphism $f$ is contained in a class which has a classifier and is closed under the formation of dependent products. Therefore, it will suffice to prove the second part of the statement.\\
Choose a class $S$ of morphisms as above, and let $t: \bar{U} \to U$ be its classifier. By universality of colimits, the induced base change functor $t^{\ast}: \mathcal{X}_{/U} \to \mathcal{X}_{/\bar{U}}$ is accessible, therefore an application of \Cref{gut} will tell us that it preserves $\kappa$-compact objects for some $\kappa$. Using \Cref{overcomp}, this means that $t$ is relatively $\kappa$-compact. Since every morphism of $S$ is a pullback of $t$, this means that $S \subseteq S_{\kappa}$, where $S_{\kappa}$ is the class of relatively $\kappa$-compact morphisms. Now, an application of \Cref{univthendep} yields that the universe we're working in is 1-inaccessible, therefore we may retrace our steps in the proof of \Cref{almostthere}, enlarging $\kappa$ if necessary and thus ensuring that $S_{\kappa}$ have a classifier and be closed under dependent products. Moreover, it is also clearly closed under dependent sums, since they are just compositions. The proof is then complete by taking $S' = S_{\kappa}$.
\end{proof}

The previous corollary shows that under the assumption of classifiability, then closure under dependent products almost implies closure under dependent sums as well. We want to conclude this writing by showing that classifiability alone, on the other hand, a priori has nothing to do with closure under either operation. In order to accomplish this, we will exhibit a class of morphisms which has a classifier but is closed under neither dependent sums nor dependent products.

\begin{ex} \label{counter}
Let us work in $\Set$. Let $\kappa$ be a cardinal which is not regular, and define $S$ as the class of all functions between sets such that all of their fibers are strictly smaller than $\kappa$. We claim that this class admits a classifier.\\
Let's consider $\kappa$ as a set. For each cardinal $\lambda < \kappa$, consider the map of sets $\lambda \to \kappa$ which is constant on the element of $\kappa$ corresponding to $\lambda$. This yields a map

\begin{equation*}
t: \coprod_{\lambda < \kappa} \lambda \to \kappa
\end{equation*}

which belongs to $S$ since, by construction, the fibers are simply all $\lambda$'s indexing the coproduct in the domain. Now consider a function $f: A \to B$ with all fibers being strictly smaller than $\kappa$. We may rewrite it as the natural map $\coprod_{b \in B} A_b \to B$, where the $A_b$'s are the fibers of $f$. We now define a map $q: B \to \kappa$ in such a way that $f$ will be a pullback of $t$ along $q$.\\
For an element $b \in B$, set $q(b)$ to be the $\lambda$ such that $|A_b| = \lambda$. An easy check will reveal that there is a pullback square

\begin{center}
\begin{tikzcd}
\coprod_{b \in B} A_b \ar[d, "f"] \ar[r] & \coprod_{\lambda < \kappa} \lambda \ar[d, "t"]\\
B \ar[r, "q"] & \kappa
\end{tikzcd}
\end{center}

where the upper horizontal map chooses a total ordering for every set $A_b$. Moreover, the map $q$ is unique with this property. Indeed, assume there is another map $q'$ having the same property. Since $q' \neq q$, there exists $b \in B$ such that $q'(b) = \lambda' \neq |A_b|$. Now we can build a double pullback diagram

\begin{center}
\begin{tikzcd}
A_b \ar[d] \ar[r] & \coprod_{b \in B} A_b \ar[d, "f"] \ar[r] & \coprod_{\lambda < \kappa} \lambda \ar[d, "t"]\\
\{ b \} \ar[r] & B \ar[r, "q'"] & \kappa
\end{tikzcd}
\end{center}

and, by pasting law, we see that $A_b$ is isomorphic to the fiber of $t$ corresponding to $\lambda'$, which is absurd, since such a fiber of $t$ should have cardinality $\lambda'$.\\
Now, observe that the class $S$ just constructed is not closed under dependent sums. For example, consider a family of sets $(X_i)_{i \in I}$ such that $|I| < \kappa$ and $|X_i| < \kappa$ but $|\coprod_{i \in I} X_i| \geq \kappa$ (this is possible since $\kappa$ has been chosen not to be regular). The canonical map $p: \coprod_{i \in I} X_i \to I$ belongs to $S$ and so does the terminal map $I \to \ast$, but their composite does not, which is immediate by the construction of these maps.\\
To show that $S$ is not closed under dependent products either, let us take the same family $(X_i)_{i \in I}$, assuming without loss of generality that none of these sets are empty or singletons. We use the formula given in \Cref{depexpr} to obtain that the dependent product the projection $p$ along the terminal map $I \to \ast$ is the (terminal map of the) set of sections of $p$, that is, $\prod_{i \in I} X_i$. This is bigger than $\coprod_{i \in I} X_i$ (see \Cref{inacc}), therefore a fortiori it doesn't belong to $S$.
\end{ex}

As a concluding observation, we want to note that, if under the assumption of classifiability we can establish that dependent products play a somewhat stronger role than dependent sums, classifiability itself does not partake in this hierarchy.\\
Indeed, the class of all morphisms in a geometric ($\infty$-)topos is obviously closed under dependent products and sums, but it does not admit a classifier, so closure under either of them does not imply classifiability. Conversely, \Cref{counter} already shows that classifiability implies neither closure. Moreover, even classifiability and dependent sums together do not imply dependent products, as it is easily seen by considering the class of relatively $\kappa$-compact morphisms in a geometric ($\infty$-)topos without the 1-inaccessibility assumption.

\section{Geometric $\subsetneq$ elementary}
We have shown in \Cref{firsthalf} that, under the assumption that the chosen Grothendieck universe is 1-inaccessible, every geometric $\infty$-topos is also a Shulman $\infty$-topos. This gives an inclusion of the class of geometric $\infty$-toposes into the class of Shulman $\infty$-toposes. We now might wonder whether this inclusion is strict. In other words, whether there exist Shulman $\infty$-toposes that are not geometric. A partial answer is given in the following, where we provide a sufficient condition for the existence of such objects. Where hitherto we used the condition of $1$-inaccessibility on the Grothendieck universe itself, now we will be using it on cardinals \emph{within} this universe, and conclude that assuming that there are enough such cardinals there exists a plethora of examples of the kind of objects we are looking for.\\
%During this section, we will be addressing a few results from literature that have not been included in the preliminary section in order not to encumber it with details that would be unnecessary to the main bulk of this writing. They are generally known by those working in the context of $\infty$-category but nonetheless, since they are often highly non-trivial, precise references will always be provided.
First, as a technical step, we present a statement which is already known for ordinary categories and certainly expected for $\infty$-categories. Namely, if a small category is either complete or cocomplete, then it is a preorder, in the sense that every homset has at most one element. The technicalities involved here are not exceedingly long or complicated, but nonetheless there's a good chance that a complete depiction of them at this stage would drive the reader's thoughts away from the main point that we want to focus on right now. For this reason, we defer the proof of the \Cref{smallcompletecat} to an appendix, trusting that it will anyway seem like a reasonable result to expect.

\begin{lem} \label{smallcompletecat}
Suppose $\mathcal{C}$ is an essentially small $\infty$-category that has either all small limits or all small colimits. Then for any choice of objects $X,Y \in \mathcal{C}$ the mapping space $\Map(X,Y)$ is contractible.
\end{lem}

\begin{lem} \label{smalltopos}
A geometric $\infty$-topos cannot be small, unless it is equivalent to a point.
\end{lem}

\begin{proof}
Suppose there is a small $\infty$-topos $\mathcal{X}$. By \Cref{smallcompletecat}, all of its mapping spaces are contractible. In particular, since for any two objects, $A,B \in \mathcal{X}$ we have $\Map(A,B) \simeq \Map(A,B) \times \Map(A,B)$, this implies that the binary coproduct of two copies of the same object $A$ is $A$ itself. Now, by \Cref{deftop} coproducts are disjoint, therefore there is a pullback square

\begin{center}
\begin{tikzcd}
\emptyset \ar[d] \ar[r] & A \ar[d]\\
A \ar[r] & A.
\end{tikzcd}
\end{center}

Obviously, replacing $\emptyset$ with $A$ will also give a pullback square, therefore $A \simeq \emptyset$, so every object is initial, and $\mathcal{X}$ is equivalent to a point.
\end{proof}

We can now finally prove the theorems that we wanted. The author does not know whether or not the following statements admit a converse, or whether there are Shulman $\infty$-toposes which do not arise as in the cases exhibited below.

\begin{prop} \label{counterexspaces}
Fix a Grothendieck universe $\mathcal{U}$. Assuming that there is a $\mathcal{U}$-small $1$-inaccessible cardinal $\kappa$, then the $\infty$-topos of spaces admits a subcategory which is a Shulman $\infty$-topos but not a geometric one.
\end{prop}

\begin{proof}
Consider the full subcategory $\mathcal{S}^{\kappa} \subseteq \mathcal{S}$ spanned by all $\kappa$-compact spaces. We will show that this is already the object that we are looking for. Observe that since $\kappa$ is inaccessible, there is a Grothendieck universe $\mathcal{U}_{\kappa} \in \mathcal{U}$ associated to it, and that $\mathcal{S}^{\kappa}$ is precisely the $\infty$-topos of spaces in the universe $\mathcal{U}_{\kappa}$. \Cref{main} now says exactly that $\mathcal{S}^{\kappa}$ is a Shulman $\infty$-topos in $\mathcal{U}_{\kappa}$, but since the definition of Shulman $\infty$-topos does not address any size issues, $\mathcal{S}^{\kappa}$ will also be a Shulman $\infty$-topos in $\mathcal{U}$.\\
We show now that it is not a geometric $\infty$-topos. Indeed, if it were, \Cref{smalltopos} would imply that it is a point, which is obviously not the case.
\end{proof}

\begin{prop} \label{counterexpresh}
Fix a Grothendieck universe $\mathcal{U}$. Given a small $\infty$-category $\mathcal{C}$, assume that there exists a $1$-inaccessible cardinal $\kappa > |\mathcal{C}|$. Then there is a subcategory of $\mathcal{P(C)}$ that is a Shulman $\infty$-topos but not a geometric one.
\end{prop}

\begin{proof}
As in \Cref{counterexspaces}, consider the full subcategory $\mathcal{P(C)}^{\kappa} \subseteq \mathcal{P(C)}$ spanned by all $\kappa$-compact objects. Again, since $\kappa$ is inaccessible, there is a Grothendieck universe $\mathcal{U}_{\kappa} \in \mathcal{U}$ associated to it. Moreover, since $\kappa > |\mathcal{C}|$, \Cref{preshcomp} says that $\mathcal{P(C)}^{\kappa}$ is a presheaf category in the smaller universe, and therefore a geometric $\infty$-topos $\mathcal{U}_{\kappa}$. As above, this implies that it is a Shulman $\infty$-topos in $\mathcal{U}$.\\
To see that $\mathcal{P(C)}^{\kappa}$ cannot be a geometric $\infty$-topos, use \Cref{counterexspaces} again.
\end{proof}

\begin{rem}
For the next proposition, we will need the notion of 2-inaccessibility. In TG set theory, for a cardinal to be 2-inaccessible just means that the set of 1-inaccessibles below it is unbounded. In the NBG axiomatic system, the assumption of a 2-inaccessible universe translates to saying that there is a proper class of 1-inaccessible cardinals, whilst the notion of 1-inaccessibility goes back to our original definition.
\end{rem}

\begin{prop}
Fix a $2$-inaccessible Grothendieck universe $\mathcal{U}$. Then every geometric $\infty$-topos $\mathcal{X}$ admits a subcategory which is a Shulman $\infty$-topos but not a geometric one.
\end{prop}

\begin{proof}
Consider a diagram

\begin{center}
\begin{tikzcd}
\mathcal{X} \ar[r, shift right, hook, "i", swap] & \mathcal{P(C)} \ar[l, shift right, "L", swap]
\end{tikzcd}
\end{center}

defining $\mathcal{X}$ as an $\infty$-topos. By uniformization, and in particular keeping \Cref{magic} in mind, we can find a cardinal $\kappa$ such that

\begin{enumerate}
\item $\kappa$ is $1$-inaccessible;
\item $\kappa > |\mathcal{C}|$;
\item $i$ and $L$ both preserve $\kappa$-compact objects;
\item $\kappa$-compact objects in $\mathcal{X}$ are stable under finite limits.
\end{enumerate}

By inaccessibility of $\kappa$, there is a Grothendieck universe $\mathcal{U}_{\kappa} \in \mathcal{U}$ associated to it. Condition (2) implies, by \Cref{preshcomp}, that $\kappa$-compact objects in $\mathcal{P(C)}$ are exactly presheaves on $\mathcal{C}$ in the smaller universe. Condition (3) means that the adjunction above restricts to

\begin{center}
\begin{tikzcd}
\mathcal{X}^{\kappa} \ar[r, shift right, hook, "i_{\kappa}", swap] & \mathcal{P(C)}^{\kappa} \ar[l, shift right, "L_{\kappa}", swap]
\end{tikzcd}
\end{center}

and, moreover, $L_{\kappa}$ preserves finite limits by condition (4). Therefore the restricted adjunction exhibits $\mathcal{X}^{\kappa}$ as a geometric $\infty$-topos in the smaller universe $\mathcal{U}_{\kappa}$. Again as in \Cref{counterexspaces} and \Cref{counterexpresh}, this implies that $\mathcal{X}^{\kappa}$ is a Shulman $\infty$-topos in $\mathcal{U}$.\\
In this case, too, using \Cref{smallcompletecat} will ensure that $\mathcal{X}^{\kappa}$ cannot be a geometric $\infty$-topos in $\mathcal{U}$.
\end{proof}

\appendix
\section{Small complete $\infty$-categories}
This appendix is dedicated to a proof of \Cref{smallcompletecat}, for which the strategy is to generalize the corresponding $1$-categorical proof, already known and spelled out, for example, in \cite{CWM}, IV, Proposition 3.\\
In order to slightly simplify the proof of its $\infty$-categorical generalization and, moreover, to emulate more closely the proof of the classical statement, we will make use of a couple of definitions, whose intuitive idea is that we are looking at some selected cells of all possible dimensions in an $\infty$-category whose behaviour is reminiscent of cells in a strict higher category, where every $n$-cell has a domain and a codomain which are $(n-1)$-cells, so that it makes sense to say when two such cells are consecutive, what other cell (in our case, not uniquely determined) is a composite of them and to describe $(n+1)$-cells having two given $n$-cells as domain and codomain respectively.\\
We want to point out that it is not impossible to prove the theorem without the machinery hereafter presented, but we believe that, once said machinery has sunk in, it will be combinatorially simpler and, more than that, it should be clear that what we are doing is really just reproposing the classical proof in a higher dimensional form.

\begin{defn}
Given a simplex $\sigma$ of dimension $n$ in a simplicial set, we know that there is a unique non-degenerate simplex $\rho$ of dimension $m \leq n$ such that $\sigma$ is a degeneracy of $\rho$. In this case, we will say that $m$ is the essential dimension of $\sigma$.
\end{defn}

\begin{defn} \label{nmorphisms}
Let $\mathcal{C}$ be an $\infty$-category. All simplices of dimension $0$ or $1$ are called $0$-morphisms and $1$-morphisms respectively. For $n \geq 2$, we say that $\alpha \in \mathcal{C}_n$ is an $n$-morphism if $\forall k = 0, \dots, n-2$ the $k$-th face of $\alpha$ is of essential dimension $k$. In this case, the $(n-1)$-th face of $\alpha$ will be called its codomain, and the $n$-th face its domain.
\end{defn}

\begin{rem}
The condition on faces given in \Cref{nmorphisms} should be interpreted as the simplex at issue being “as degenerate as possible” once given the last two faces, thus giving in practice no information other than a domain, a codomain and an arrow going from the former to the latter, just as $n$-morphisms in a strict higher category.\\
Also, a brief inspection of an $n$-morphism in an $\infty$-category will reveal that both the domain and the codomain of an $n$-morphism are $(n-1)$-morphisms, and that they in turn have the same $(n-2)$-morphisms as domain and codomain. This allows us to talk about $n$-fold domains and codomains, obtained by taking either one $n$ times, thus decreasing the dimension by $n$, just as it is done in strict higher categories or in globular sets.
\end{rem}

The following construction is straightforward but very important for what follows. Given two objects $X,Y$ in an $\infty$-category $\mathcal{C}$, the $0$-simplices of $\Hom^L(X,Y)$ (defined for instance in \cite{HTT}, Remark 1.2.2.5, but an equally valid discussion can be done using $\Hom^R(X,Y)$ instead) are precisely $1$-morphisms having $X$ as domain and $Y$ as codomain, while $1$-simplices from $f$ to $g$ are precisely $2$-morphisms having $f$ as domain and $g$ as codomain. Now we can regard $\Hom^L(X,Y)$ as an $\infty$-category and look at $\Hom^L(f,g)$ whenever $f$ and $g$ are $0$-simplices of it. Similarly as above, $0$-simplices of $\Hom^L(f,g)$ are $2$-morphisms having $f$ as domain and $g$ as codomain, and given two such $\alpha$ and $\beta$, a $1$-simplex from $\alpha$ to $\beta$ is precisely a $3$-morphism having $\alpha$ as domain and $\beta$ as codomain.\\
Iterating this procedure, we obtain that for two $n$-morphisms $\phi$ and $\theta$ the $n$-fold $\Hom^L$ construction gives a mapping space $\Hom^L(\phi, \theta)$ whose $0$-simplices are precisely $(n+1)$-morphisms from $\phi$ to $\theta$ and whose $1$-simplices are precisely $(n+2)$-morphisms between those.

\begin{defn} \label{homotopic}
Given two $n$-morphisms $\alpha$ and $\beta$ in an $\infty$-category, we say that they are homotopic if there is a $(n+1)$-morphism having $\alpha$ as domain and $\beta$ as codomain (in particular, $\alpha$ and $\beta$ necessarily have the same domain and codomain themselves). As can be checked using horn fillings, this defines an equivalence relation on $n$-morphisms.
\end{defn}

\begin{lem}
Suppose $\mathcal{C}$ is an essentially small $\infty$-category, and that it has either all small limits or all small colimits. Then for any choice of objects $X,Y \in \mathcal{C}$ the mapping space $\Map(X,Y)$ is contractible.
\end{lem}

\begin{proof}
We will treat the case of completeness, the other is perfectly dual. Clearly the statement is invariant under equivalence, therefore we may assume that $\mathcal{C}$ is strictly small. We will prove that for any $n \geq 0$, any map $\partial\Delta^n \to \Map(X,Y)$ can be extended to a map $\Delta^n \to \Map(X,Y)$. Observe that if two Kan complexes are weakly equivalent, then one has the right lifting property against all inclusions $\partial \Delta^n \subset \Delta^n$ if and only the other does, because for Kan complexes this is the same as saying that the terminal map is a trivial fibration. Therefore, we are free to switch between different models for the mapping spaces involved, whenever it comes in handy. In particular, we will keep in mind the three homotopy equivalent objects mentioned in \cite{HTT}, Corollary 4.2.1.8.\\
Assume by contradiction that there exists a map $q: \partial \Delta^n \to \Hom^L(X,Y)$ that admits no extension to $\Delta^n$. This gives a non-null element of $\pi_{n-1}(\Hom^L(X,Y))$, which is isomorphic to $\pi_{n-2}(\Omega\Hom^L(X,Y))$, where the loop space is calculated taking $f = q(0)$ as base point. Therefore $\Omega\Hom^L(X,Y) = \Fun(\Delta^1,\Hom^L(X,Y))_{0,1 \mapsto f}$. Using \cite{HTT}, Corollary 4.2.1.8, we can replace this by $\Hom^L(f,f)$, therefore finding a non-null element of $\pi_{n-2}(\Hom^L(f,f))$. Repeating this procedure $n-1$ times, we obtain a non-null element of $\pi_0(\Hom^L(\phi,\phi))$, for some $(n-1)$-morphism $\phi$. The fact that it is non-null means that there are two $n$-morphism that are not homotopic, in the sense of \Cref{homotopic}. In particular, we have proven that, once fixed a base point $f$, the elements of $\pi_n(\Hom^L(X,Y))$ are in bijection with homotopy classes of $n$-morphisms having a specific degeneracy of $f$ as domain and codomain.\\
Now, let $\mathsf{H}_n$ be the quotient of the set of $n$-morphisms of $\mathcal{C}$ under the homotopy equivalence relation. Since this set is small and $\mathcal{C}$ is complete, we can consider the product $\prod_{\mathsf{H}_n} Y$, therefore we have an equivalence of mapping spaces $\Hom^L(X, \prod_{\mathsf{H}_n} Y) \simeq \prod_{\mathsf{H}_n} \Hom^L(X,Y)$. The discussion above says that there are at least two equivalence classes of $n$-morphisms having $X$ and $Y$ as $n$-fold domain and codomain, in particular two distinct elements of $\pi_n(\Hom^L(X,Y))$, therefore, since the functor $\pi_n$ preserves products, $\pi_n(\Hom^L(X,\prod_{\mathsf{H}_n} Y)) \cong \pi_n(\prod_{\mathsf{H}_n} \Hom^L(X,Y))$ has at least $2^{|\mathsf{H}_n|}$ distinct elements. Again the discussion above allows to reduce the analysis to equivalence classes of $n$-morphisms only, therefore giving an amount of such which is strictly bigger than the number of all equivalence classes of $n$-morphisms. We obtained thus the desired contradiction.
\end{proof}

%_________________________________________________%

\bibliographystyle{alpha}
\bibliography{bibliography}

\begin{thebibliography}{GHN15}

\bibitem[AR94]{LocPresAccCat}
Ji\v{r}í Adámek and Ji\v{r}í Rosicky.
\newblock {\em Locally Presentable and Accessible Categories}.
\newblock Cambridge University Press, 1994.

\bibitem[GHN15]{GepnerHaugsengNikolaus}
David Gepner, Rune Haugseng, and Thomas Nikolaus.
\newblock Lax colimits and free fibrations in $\infty$-categories.
\newblock \url{https://arxiv.org/abs/1501.02161v2}, 2015.

\bibitem[Jec06]{Jech}
Thomas Jech.
\newblock {\em Set Theory}.
\newblock Springer Verlag, 2006.

\bibitem[Lan71]{CWM}
Saunders~Mac Lane.
\newblock {\em Categories for the Working Mathematician}.
\newblock Springer Verlag, 1971.

\bibitem[Lur09]{HTT}
Jacob Lurie.
\newblock {\em Higher Topos Theory}.
\newblock Princeton University Press, 2009.

\bibitem[Mon18]{mythesis}
Giulio~Lo Monaco.
\newblock On two extensions of the notion of $\infty$-topos.
\newblock 2018.

\end{thebibliography}
\end{document}